\documentclass
{amsart}
\usepackage{amssymb,amsmath,amsthm,bm,color}
\usepackage{ upgreek }

\usepackage
{hyperref}
\usepackage[nameinlink]{cleveref}
\usepackage[mathscr]{eucal}
\usepackage[all,arc]{xy}
\usepackage{tikz}
\usepackage{tikz-cd}

\allowdisplaybreaks[1]   
\theoremstyle{definition}
\newtheorem{definition}{Definition}[section]
\theoremstyle{plain}
\newtheorem{theorem}{Theorem}[section]
\newtheorem{lemma}{Lemma}[section]
\newtheorem{proposition}{Proposition}[section]
\newtheorem{corollary}{Corollary}[section]
\newtheorem{conjecture}{Conjecture}
\newtheorem{question}{Question}
\theoremstyle{remark}
\newtheorem{remark}{Remark}[section]
\makeatletter

\newcommand{\on}{\operatorname}
\newcommand{\bra}{\langle}
\newcommand{\ket}{\rangle}

\newcommand{\mf}{\mathfrak}
\newcommand{\mc}{\mathcal}
\newcommand{\ra}{\rightarrow}
\newcommand{\id}{\text{Id}}
\newcommand{\Hom}{\text{Hom}}
\newcommand\xr{\xrightarrow}
\newcommand{\g}{\mathfrak{g}}
\newcommand{\s}{\mathfrak{s}}
\newcommand{\affg}{\widehat{\mathfrak{g}}}
\newcommand{\h}{\mathfrak{h}}
\newcommand{\W}{\mathscr{W}}
\newcommand{\Z}{\mathbb{Z}}
\newcommand{\Q}{\mathbb{Q}}

\newcommand{\C}{\mathbb{C}}
\newcommand{\bH}{\mathbb{H}}
\newcommand{\ag}{\widehat{\g}}

\newcommand{\cM}{\mathcal{M}}
\newcommand{\cY}{\mathcal{Y}}

\newcommand{\cC}{\mathcal{C}}
\newcommand{\cCloc}{\mathcal{C}^{\text{loc}}}
\newcommand{\cD}{\mathcal{D}}

\newcommand{\cN}{\mathcal{N}}
\newcommand{\lam}{\lambda}
\newcommand{\Lam}{\Lambda}
\newcommand{\+}{\oplus}
\renewcommand{\*}{\otimes}

\newcommand{\comm}[2]{\mathrm{Com}(#1,#2)}
\newcommand{\vir}{\mathrm{Vir}}
\newcommand{\svir}{\mathrm{sVir}_0}
\newcommand{\sVir}{\mathrm{sVir}}

\newcommand{\FP}{\mathrm{FP}}
\newcommand{\cw}{\mathrm{cw}}

\newcommand{\Mod}[1]{\mathsf{#1}}
\newcommand{\fus}{\boxtimes}

\newcommand{\alg}[1]{\mathsf{#1}}
\newcommand{\aA}{{\alg{A}}}

\newcommand {\cU}{\mathcal{U}}

\newcommand{\cF}{\mathcal{F}}
\newcommand{\cG}{\mathcal{G}}
\newcommand {\cW}{\mathcal{V}}
\newcommand{\mU}{{\Mod{U}}}
\newcommand{\mW}{{\Mod{V}}}
\newcommand{\HHom}{\mathrm{Hom}}

\begin{document}

\title{On lisse non-admissible minimal  and principal W-algebras}

\author[T~Arakawa]{Tomoyuki Arakawa}
\address[Tomoyuki Arakawa]{
Research Institute for Mathematical Sciences\\
Kyoto University\\
Kyoto, Japan, 606-8502.
}
\email{arakawa@kurims.kyoto-u.ac.jp}
\thanks{T.A is partially supported by JSPS Kakenhi Grant numbers 21H04993 and 19KK0065}

\author[T~Creutzig]{Thomas Creutzig}
\address[Thomas Creutzig]{
Department of Mathematical and Statistical Sciences\\
University of Alberta\\
Edmonton, Canada, AB T6G 2G1.
}
\email{creutzig@ualberta.ca}
\thanks{T.C.~is supported by a NSERC Discovery Grant}

\author[K~Kawasetsu]{Kazuya Kawasetsu}
\address[Kazuya Kawasetsu]{
Priority Organization for Innovation and Excellence\\
Kumamoto University\\
Kumamoto, Japan, 860-8555.
}
\email{kawasetsu@kumamoto-u.ac.jp}
\thanks{K.K. is partially supported by MEXT Japan ``Leading Initiative for Excellent Young Researchers (LEADER)'',
JSPS Kakenhi Grant numbers 19KK0065, 21K13775 and 21H04993.
}

\begin{abstract}
We discuss a possible generalization of a result by the third-named author on the rationality of non-admissible minimal W-algebras. 
We then apply this generalization to finding rational non-admissible principal W-algebras.
\end{abstract}

\maketitle


\section{Introduction}   

Let $\g$ be a simple Lie algebra, 
$f$ a nilpotent element of $\g$,
and $k\neq -h^\vee$ a non-critical number, 
where $h^{\vee}$ is the dual Coxeter number of $\g$.
 Let $\W^k(\g,f)$ and $\W_k(\g,f)$ denote the universal and simple {\em W-algebras}, respectively \cite{KacRoaWak03}. 
 The {\em principal W-algebras} $\W^k(\g)=\W^k(\g,f_{prin})$ with principal nilpotent elements $f=f_{prin}$ coincide with the W-algebras defined in \cite{FF90}. 

It is one of the central problems in the study of W-algebras to classify lisse and rational simple W-algebras $\W_k(\g,f)$. 
 Frenkel, Kac, and Wakimoto \cite{FKW92} conjectured 
 the  existence and construction of {\em minimal series W-algebras}, that is,
 the lisse, rational principal W-algebras that generalize the minimal series Virasoro vertex algebras.
 The conjecture \cite{FKW92} 
 was proved 
 by the first-named author \cite{Ara09b,A2012Dec};
 The minimal series W-algebras are  obtained by non-degenerate admissible affine vertex algebras through quantum Hamiltonian reduction. 
The conjecture of Frenkel, Kac, and Wakimoto was generalized by Kac and Wakimoto \cite{KacWak08} to non-principal W-algebras. This was further generalized by the first named author \cite{Ara09b} to a conjecture stating that each admissible affine vertex algebra produces exactly one lisse, rational W-algebra, by performing the quantum Hamiltonian reduction associated with the unique nilpotent orbit that is open in the associated variety \cite{Ara12} of the admissible affine vertex algebra. At this point, it was believed that all lisse, rational W-algebras should be {\em admissible}, that is, they should come from admissible affine vertex algebras in this manner \cite{KacWak08}.


While the conjecture of Kac, Wakimoto, and the first named author was proved by \cite{AEkeren19, Fas22, McRaeRational}, the third-named author \cite{Kawa15} showed that there exist in fact lisse, rational {\em non-admissible} W-algebras. More precisely, it was shown in \cite{Kawa15} that for each member $\g$ of {\em Deligne's exceptional series} \cite{De96}
\begin{align*}
a_1 \subset a_2 \subset g_2 \subset d_4 \subset f_4 \subset e_6 \subset e_7 \subset e_8,
\end{align*} the minimal W-algebra $\W_{-h^\vee/6}(\g, f_\theta)$ is lisse and rational, where $h^\vee$ is the dual Coxeter number of $\g$. For $\g = d_4, e_6, e_7, e_8$, the level $-h^\vee/6$ is a negative integer, and hence $\W_{-h^\vee/6}(\g, f_\theta)$ is not at an admissible level.

After the work \cite{Kawa15},
it was shown in  \cite{AM15}  that, for a member $\g$ of Deligne's exceptional series,
 the minimal W-algebra $\W_{k(m)}(\g, f_\theta)$ 
at level
 \begin{align}
k(m)=-\frac{h^{\vee}}{6}+m
\end{align}
is lisse for
 all integer $m$ that is equal to or greater than $-1$.
Again, 
for $\g = d_4, e_6, e_7, e_8$, 
the lisse W-algebra $\W_{k(m)}(\g, f_\theta)$ is not at an admissible level.
Since $\W_{k(-1)}(\g, f_\theta)=\C$ \cite{AM15}
and $\W_{k(0)}(\g, f_\theta)$ is rational \cite{Kawa15} as explained above,
it is natural to ask the following question.
\begin{question}\label{question}
Let $\g$ be 
a simple Lie algebra of type $d_4$, $e_6$, $e_7$, $e_8$,
and let $m$ be an  integer greater than or equal to $-1$.
Is the lisse mimimal W-algebra $\W_{k(m)}(\g,f_{\theta})$ is rational?
\end{question}

Sun \cite{Sun} recently gave an affirmative answer\footnote{Sun \cite{Sun} verified the rationality only at the level
of characters, but it is  possible to prove the rationality rigorously.} to 
 \Cref{question} for $m=1$ and $\g = e_6, e_7, e_8$,  by making use of  Hecke operators.

 \begin{conjecture}\label{conj:minimal}
Let $\g$, $m$ be as in \Cref{question}.
The simple minimal $W$-algebra  $\W_{k(m)}(\g,f_{\theta}) $ is   rational if and only if 
$k(m)<0$, that is, $m=-1,0$ for $d_4$, $m=-1,0,1$ for $e_6$, 
$m=-1,0,1,2$ for $e_7$ and $m=-1,0,1,2,3,4$ for $e_8$.

\end{conjecture}

In particular,
 we conjecture that, for  $\g = d_4, e_6, e_7, e_8$,  the minimal W-algebra $\W_{k}(\g, f_\theta)$ is lisse and non-rational for a non-negative integer $k$,
 which would provide a new family of 
lisse {\em  logarithmic} vertex algebras.

 \smallskip

Let us mention that
\Cref{conj:minimal} can be used to find new lisse, rational principal W-algebras as well, due to the following assertion.

\begin{theorem}\label{sec:dualsimple}
Let $\g$, $k(m)$, be as above with $m\geq 0$
such that $k(m)<0$.
Define the number  $\ell(m)$ by
\begin{equation*}
\frac{1}{k(m)+h^\vee}+\frac{1}{\ell(m)+h^\vee}=1.
\end{equation*}
There is an embedding
$$\W_{k(m)}(\g,f_{\theta})\*\W_{\ell(m)}(\g)\hookrightarrow \W_{k(m-1)}(\g,f_{\theta})\* L_1(\g)$$
of conformal vertex algebras,
where  $L_1(\g)$ is equipped with the {\em Urod conformal vector} \cite{AraCreFei22}.
Moreover,  
$\W_{k(m)}(\g,f_{\theta})$ and $\W_{\ell(m)}(\g)$ form a dual pair in 
$\W_{k(m-1)}(\g,f_{\theta})\* L_1(\g)$.
\end{theorem}
Here vertex subalgebras $W_1$ and $W_2$ of a vertex algebra $V$ are said to form a dual pair
if $\on{Com}(W_1,V)=W_2$ and $\on{Com}(W_2,V)=W_1$,
where 
$$\on{Com}(W,V)=\{w\in V\mid [v_{(m)},w_{(n)}]=0\text{ for all }m,n\in \Z,\ v\in V\}.$$

The following assertion is widely believed.
\begin{conjecture}\label{conj:coset}
Let $V$ be a vertex algebra,
$W$ a vertex subalgebra of $V$.
Suppose that $V$ and $W$ are finitely strongly generated, conformal, rational and lisse.
Then $\on{Com}(W,V)$ is also rational and lisse.
\end{conjecture}
\Cref{sec:dualsimple},
\Cref{conj:coset}
and  \Cref{conj:minimal}
implies:
\begin{conjecture}\label{conj:principal}
Let $\g$ be 
a simple Lie algebra of type $d_4$, $e_6$, $e_7$, $e_8$,
and let $m$ be a non-negative integer.
The simple principal $W$-algebra $\W_{\ell(m)}(\g)$ is rational and lisse
if $k(m)<0$.
\end{conjecture} 
Note that all the levels $\ell$ appearing in  \Cref{conj:principal} (2)
are not admissible for $\widehat{\g}$,
see Table \ref{table: ell(0)}.
\begin{table}[hbtp]
  \caption{The level of  principal W-algebras $\W_{\ell(m)}(\g)$ appearing in  \Cref{conj:principal}}
  \label{table: ell(0)}
  \begin{tabular}{c|c|c|c|c}
      & $d_4$ &  $e_6$ &$e_7$& $e_8$ \\
    \hline 
    $\ell(0)+h^{\vee}$ & $5/4$  & $10/9$& $15/14$     &$25/24$\\
     $\ell(1)+h^{\vee}$ &  & $11/10$& $16/15$     &$26/25$\\
       $\ell(2)+h^{\vee}$ &  & & $17/16$     &$27/26$\\
        $\ell(3)+h^{\vee}$ &  & &    &$28/27$\\
          $\ell(4)+h^{\vee}$ &  & &      &$29/28$
  \end{tabular}
 \end{table}
 
 \begin{theorem}\label{theorem m=0}
 \Cref{conj:principal} is true for $m=0$.
 \end{theorem}

\Cref{theorem m=0}
was conjectured by \cite{LiXieYan23},
see the first row of Table 11 in loc.cit and use the Feigin-Frenkel duality.

\section{Proof of  \Cref{th:new-rational}.}
\label{sec:deligne1}
Let $V^k(\g)$ be the universal affine vertex algebra associated to $\g$ at level $k$,
$L_k(\g)$ the unique simple graded quotient of $V^k(\g)$.


The central charge of the principal W-algebra $\W^k(\g)$
is given by
\begin{align}\label{eq:cc-of-principal-W}
-\on{rk}(\g)\frac{
((h+1)(k+h^{\vee})-h^{\vee})(r^{\vee}{}^Lh^{\vee}(k+h^{\vee})-(h+1))}
{k+h^{\vee},}
\end{align}
where  $\on{rk}(\g)$ is the rank of $\g$,
$h$ is the Coxeter number of $\g$,
${}^L\g^{\vee}$ is the dual Coxeter number of the Langlands dual of
$\g$,
and $r^{\vee}$ is the lacing number of $\g$.

We set
$$
\g=d_4,e_6,e_7,e_8,\quad k=k(0)=-h^\vee/6,\quad \ell=\ell(0)=\frac { 5h^\vee } { 5h^\vee-6 }-h^\vee.
$$
Then, $k=-1,-2,-3,-5$ and $\ell+h^\vee=5/4,10/9,15/14,25/24$, respectively.

Let $\on{Vir}^c$ be the universal Virasoro vertex algebra at central charge $c$,
$\on{Vir}_{c}$ the unique simple quotient of $\on{Vir}^c$.
It is known \cite{Wan93} that $\on{Vir}_{c}$ is rational and lisse if and only if 
\begin{align*}
c_{p,q}=1-\frac{6(p-q)^2}{pq}
\end{align*}
for  $p,q\in \Z_{\geq 2}$,
$(p,q)=1$.

It is clear that the following assertion
proves
 \Cref{conj:principal} for $m=0$.

\begin{theorem}\label{th:new-rational}
Let $\g$ be 
a simple Lie algebra of type $d_4$, $e_6$, $e_7$, $e_8$.
We have
\begin{align*}
\W_{\ell(0)}(\g)\cong \on{Vir}_{c_{2.5}}.
\end{align*}
\end{theorem}


\begin{proof}
The central charge of $\W_{\ell}(\g)$ is $-22/5$, which coincides with that of $\on{Vir}_{c_{2,5}}$.
Thus we have a conformal vertex algebra homomorphism
$$\varphi:\on{Vir}^{c_{2,5}}\rightarrow \W_{\ell}(\g).$$

We wish to show that $\on{im}\varphi \cong \on{Vir}_{c_{2,5}}$.
Since $\on{Vir}^{c_{2,5}}$ has length two,
we have either $\on{im}\varphi \cong \on{Vir}^{c_{2,5}}$
or $\on{im}\varphi \cong \on{Vir}_{c_{2,5}}$.
So we suppose that $\on{im}\varphi \cong \on{Vir}^{c_{2,5}}$
and obtain a contradiction.

By Theorem \ref{sec:dualsimple}
and the fact that $\W_{k-1}(\g,f_\theta)\cong \C$ (\cite{AM15}),
we have the embedding of vertex algebras
\begin{align}\label{eq:m=0embedding}
\W_{k}(\g,f_{\theta})\*\W_{\ell}(\g)\hookrightarrow 
L_1(\g),
\end{align}
which is  conformal 
when $L_1(\g)$ is equipped with the Urod conformal vector.
By \cite{Kawa15}, 
one knows that the simple minimal W-algebra
$\W_{k}(\g,f_\theta)$ is lisse and rational.
In particular,
it admits an asymptotic datum \cite{AraEkeMor23}.
 Since changing the conformal vector to the Urod conformal vector does not affect the asymptotic behavior of the character,
 we obtain that
 \begin{align}\label{eq:growth0}
 \mathbf{g}_{L_1(\g)}\geq \mathbf{g}_{\W_k(\g,f_{\theta})}+\mathbf{g}_{\on{Vir}^{c_{2,5}}}=\mathbf{g}_{\W_k(\g,f_{\theta})}+1,
 \end{align}
 where $\mathbf{g}_V$ denotes
 the asymptotic growth \cite{AraEkeMor23} of $V$.
Here we have used the fact 
that 
$\mathbf{g}_{\on{Vir}^{c}}=1$.

By \cite{Kawa15}, 
one knows that
there is an conformal embedding
$$\on{Vir}_{c_{3,5}}\otimes L_1(\g)^{\s[t]}\hookrightarrow \W_k(\g,f_{\theta}),$$
where 
$\s$ is the $\mathfrak{sl}_2$-triple  $\langle e_\theta,\theta,f_\theta\rangle\subset \g$
associated with  the highest root $\theta$ of $\g$
and $L_1(\g)^{\s[t]}=\{v\in L_1(\g)\mid xt^n v=0\ \forall x\in \mf{s}, \ n\geq 0\}=\comm{L_1(\s)}{L_1(\g)}$.
It was shown in  \cite{Kawa15} that $L_1(\g)^{\s[t]}$ is rational and lisse by
determining it explicitly.
It follows that
the vertex algebras 
$\on{Vir}_{c_{3,5}}\otimes L_1(\g)^{\s[t]}\subset  \W_k(\g,f_{\theta})$ and
 $L_1(\mf{s})\* L_1(\g)^{\s[t]} \subset L_1(\g)$ are rational, and satisfy the assumptions 
 in \cite[Proposition2.4]{AraEkeMor23}.
%
Therefore, we have
\begin{align}
\mathbf{g}_{\on{Vir}_{c_{3,5}}}+\mathbf{g}_{L_1(\g)^{\s[t]}}=\mathbf{g}_{\W_k(\g,f_{\theta})}, \label{eq:growth1}\\
\mathbf{g}_{L_1(\mf{s})}+\mathbf{g}_{L_1(\g)^{\s[t]}}=\mathbf{g}_{L_1(\g)}.
\label{eq:growth2}
\end{align}
On the other hand,
one knows  form \cite{BerFeiLit16} that
$L_1(\mf{s})=L_1(\mf{sl}_2)$ with a Urod conformal vector is a conformal extension of $\on{Vir}_{c_{2,5}}\otimes \on{Vir}_{c_{3,5}}$.
Hence we have
\begin{align}
\mathbf{g}_{\on{Vir}_{c_{2,5}}}+\mathbf{g}_{\on{Vir}_{c_{3,5}}}=\mathbf{g}_{L_1(\mf{s})}.
\label{eq:growth3}
\end{align}
From
\eqref{eq:growth1},
\eqref{eq:growth2},
and \eqref{eq:growth3},
we obtain that
\begin{align*}
\mathbf{g}_{L_1(\g)}=\mathbf{g}_{\W_k(\g,f_{\theta})}+\mathbf{g}_{\on{Vir}_{c_{2,5}}}.
\end{align*}
But then \eqref{eq:growth0}
implies that 
$1\leq \mathbf{g}_{\on{Vir}_{c_{2,5}}}=1-\frac{6}{2\cdot 5}$, which is a contradiction.

We have shown that $\on{im} \varphi\cong \on{Vir}_{c_{2,5}}$, and thus,
$\W_{\ell}(\g)$ decomposes into a direct sum of simple $ \on{Vir}_{c_{2,5}}$-modules.
However, $ \on{Vir}_{c_{2,5}}$ is the only simple module over itself that has an integral conformal dimension.
Therefore,
we conclude that $\W_{\ell}(\g)\cong  \on{Vir}_{c_{2,5}}$ as required.
\end{proof}
\begin{remark}
Let us consider $\g=a_1, a_2$ 
with $\ell+h^\vee=5h^\vee/(5h^\vee-6)=5/2,5/3$, $\g=g_2$ with $\ell+4=7/15$ or $5/7$ 
and $\g=f_4$ with $\ell+9=13/20$ or $10/13$.
Then we can also prove $\W_\ell(\g)\cong  \on{Vir}_{c_{2,5}}$ by using growths of characters since in these cases, $\W_\ell(\g)$ are rational minimal models 
(\cite{A2012Dec})
and their characters are already known.
We thus observe a uniform phenomenon that $\W_\ell(\g)\cong  \on{Vir}_{c_{2,5}}$ for
 the Deligne exceptional series.
%
\end{remark}

\section{Linkage of principal  W-algebras}
\label{section:Linkage}
For a weight $\lam$ of $\affg$,
let $M(\lam)$ be the Verma module of $\affg$ if highest weight $\lam$,
$L(\lam)$ the unique simple quotient of $M(\lam)$.
Let $[M(\lam): L(\mu)]$ be the multiplicity of $L(\lam)$ in the local composition factor of $M(\lam)$.
By \cite{KacKaz79},
$[M(\lam):L(\mu)]\ne 0$
if and only if
$\mu \preceq \lam$,
that is,
 there exists a sequence $\{\beta_i\}$ of  positive roots  of $\affg$,
a sequence $\{n_i\}$ of  positive integers 
and a sequence of weights $\{\lam_i\}$ such that $\lam_0=\lam$,
$\lam_r=\mu$,
$\lam_i=\lam_{i-1}-n_i \beta_i$,
$2(\beta_i,\lam_{i-1}+\hat{\rho})= n_i(\beta_i,\beta_i)$.

Let us describe the linkage of principal W-algebras
obtained in \cite{Ara07},
see also \cite{Dhi21}.
Recall that the Zhu algebra $\on{Zhu}(\W^k(\g))$
is isomorphic
to the center $\mc{Z}(\g)$ of $U(\g)$ (\cite{Ara07}).
For $\lam\in \h^*$
\begin{align*}
\chi_{\lam}:\on{Zhu}(\W^k(\g))=\mc{Z}(\g)\ra \C
\end{align*}
be  the  central character as in \cite[(27)]{ACL17},
so that
\begin{align}
\chi_{\lam}=\chi_{\mu}\iff 
\mu\in W\circ_k \lam,
\end{align}
where
$W$ is a Weyl group of $\g$
and
\begin{align}
w\circ_k \lam=w(\lam+\rho-(k+h^{\vee})\rho^{\vee})-(\rho-(k+h^{\vee})\rho^{\vee})
\end{align}
for $w\in W$.
Here,
$\rho=1/2\sum_{\alpha\in \Delta_+}\alpha$,
$\rho^{\vee}=1/2\sum_{\alpha\in \Delta_+}\alpha^{\vee}$,
and $\Delta_+$ is the set of positive roots of $\g$.

Let $\mathbf{M}_k(\chi_{\lam})$
be the Verma module 
of $\W^k(\g)$ with highest weight $\chi_{\lam}$,
$\mathbf{L}_k(\chi_{\lam})$ the unique simple graded quotient of $\mathbf{M}_k(\chi_{\lam})$,
see \cite{Ara07} for the definition.
We have
$\W_k(\g)\cong \mathbf{L}_k(\chi_0)$.

Let $k\ne -h^{\vee}$,
so that $\W^k(\g)$ is conformal.
Let $\omega_W$ be the conformal vector of $\W^{\ell}(\g)$
and set $L^W(z)=\sum_{n\in \Z}L_n^Wz^{-n-2}=Y(\omega_W,z)$.
Let $M_{d}^{gen}$ and $M_{d}$ be
the generalized $L_0^W$-eigenspace 
and the  $L_0^W$-eigenspace of eigenvalue $d$
of a $\W^{\ell}(\g)$-module $M$, respectively.

We have
$\mathbf{M}_{k}(\chi_{\lam})=\bigoplus_{d\in h_{\lam}+\Z_{\geq 0}}\mathbf{M}_{k}(\chi_{\lam})_{d}$
and
\begin{align*}
\on{tr}_{\mathbf{M}_{k}(\chi_{\lam})}q^{L_0^W}=\frac{q^{h_{\lam}}}{\prod_{j\geq 1}(1-q^j)^{\on{rk}\g}},
\end{align*}
where
\begin{align}
h_{\lam}=\Delta_{\lam}-(\lam|\rho^{\vee}),
\quad 
\Delta_\lam=\frac{(\lam|\lam+2\rho)}{2(k+h^{\vee})}.
\end{align}

Let $\mc{O}(\W^k(\g))$ be the full subcategory
of the category of
$\W^k(\g)$-modules consisting of modules $M$
such that $M=\bigoplus_{d}M_d^{gen}$,
$\dim M_d^{gen}<\infty$ for all $d$,
and
there exists finitely many $d_1,\dots, d_r\in \C$
such that $M_d^{gen}=0$ unless $d\in \bigcup_{i}d_i+\Z_{\geq 0}$.
Both $\mathbf{M}_k(\chi_{\lam})$
and $\mathbf{L}_k(\chi_{\lam})$
are  objects of $\mc{O}(\W^k(\g))$,
and $\{\mathbf{L}_k(\chi_{\lam})\mid \lam\in \h^*/W\circ_k \}$ gives a complete representative
of simple objects of the abelian category $\mc{O}(\W^k(\g))$.
Let 
$[\mathbf{M}_k(\chi_{\lam}):\mathbf{L}_k(\chi_{\mu})]$
be 
the multiplicity
of $\mathbf{L}_k(\chi_{\mu})$
in the local composition factor of $\mathbf{M}_k(\chi_{\lam})$.

We define a partial ordring
$\preceq_{k} $ on the set 
$ \h^*/W\circ_k=\on{Specm}(\on{Zhu}(\W^k(\g)))$
by setting
$\chi_{\mu} \preceq_{k} \chi_{\lam}$ 
if and only if
there exist $\lam'\in W\circ_k \lam$,
$\mu'\in W\circ_k \mu$
such that
$\mu'$ is anti-dominat
and
\begin{align*}
\mu'-(k+h^{\vee})\rho^{\vee}+k\Lam_0-\Delta_{\mu'-(k+h^{\vee})\rho^{\vee}}\delta\preceq \lam'-(k+h^{\vee})\rho^{\vee}+k\Lam_0
-\Delta_{\lam'-(k+h^{\vee})\rho^{\vee}}\delta.
\end{align*}
Note that
\begin{align}
h_{\lam}=\Delta_{\lam-(k+h^{\vee})\rho^{\vee}}+(\rho|\rho^{\vee})-\frac{k+h^{\vee}}{2}|\rho^{\vee}|^2.
\label{eq:conformal-dim-comparison}
\end{align}

The following assertion is a direct consequence of \cite[Main Theorem 1]{Ara07}
and the linkage principal of $\affg$-modules \cite{KacKaz79}.
\begin{theorem}
For  $\lam,\mu\in \h^*$
the following conditions are equivalent.
\begin{enumerate}
\item 
$[\mathbf{M}_k(\chi_{\lam}):\mathbf{L}_k(\chi_{\mu})]\ne 0$;
\item 
$\chi_{\mu}\preceq_k \chi_\lam$.\end{enumerate}
\end{theorem}

Define  an equivalence relation $\sim_k$ in $\lam\in \h^*/W\circ$
by setting
$\chi_{\lam}\sim \chi_{\lam'}$ if there exists $\mu\in \h^*$
such that 
$\chi_{\lam}, \chi_{\lam' }\preceq_k \chi_{\mu} $
or 
$\chi_{\lam}, \chi_{\lam' }\succeq_k \chi_{\mu} $.

For $\lam\in \h^*/W\circ$,
let $\mc{O}(\W^k(\g))^{[\chi_{\lam}]}$ be the block
corresponding to $[\chi_{\lam}]$,
that is the full suncategory of
$\mc{O}(\W^k(\g))$ consisting of objects
$M$ such that $[M: \mathbf{L}_k(\chi_{\mu})]= 0$
unless $\chi_{\mu}\sim \chi_{\lam}$.
Then
we have
$\mc{O}(\W^k(\g))=\bigoplus_{[\chi_{\lam}]} \mc{O}(\W^k(\g))^{[\chi_{\lam}]}$.
For $M\in \mc{O}(\W^k(\g))$,
let $M=\bigoplus M^{[\chi_\lam]}$ be the corresponding decomposition.


\section{Proof of \Cref{sec:dualsimple} }
\label{secproofthm21}
\begin{theorem}\label{thm:simplicity-of-cosets}
Let $V=\bigoplus_{d\geq 0}V_d$ be a non-negatively graded conformal vertex algebra 
with $V_0=\C$, $\dim V_d<\infty$ for all $d$,
equipped with a non-degenerate 
symmetric invariant bilinear form  $(~|~)$ in the sense of \cite{FHL93}.
Suppose that 
$V$ is equipped with a compatible $\affg$-structure \cite[7.1.3]{FreBen04}
with level $k\not\in -h^{\vee}-\Q_{\leq 0}$,
that is,
there exists a vertex algebra homomorphism
$\varphi:V^k(\g)\ra V$ such that
\begin{enumerate}
\item $\varphi(x_{(-1)}|0\rangle)$ is primary of conformal weight $1$ for all   $x\in \g$;
\item  $k+h^{\vee}\not\in \Q_{\leq 0}$.
\end{enumerate}
Then the commutant $\on{Com}(\varphi(V^k(\g),V))=V^{\g[t]}$ is simple.
\end{theorem}
\begin{proof}
Set $U=\on{Com}(\varphi(V^k(\g),V))$
and
 $L_n=\omega_{(n+1)}$,
where $\omega$ is the conformal vector of $V$.
By the assumption $(2)$,
$V^k(\g)$ is conformal with the Sugawara conformal vector $\omega_{\g}$.
Set $L_n^\g=\varphi(\omega_{\g})_{(n+1)}$.
By the assumption $(1)$,
$\omega_{\g}\in V_2$
and $L_1 \varphi(\omega_\g)=0$.
Hence
by 
\cite[Theorem 3.11.12 and Remark 3.11.13]{LepLi04},
$L_n=L_n^\g$ on $\varphi(V^k(\g))$
and $L_n=\omega''_{(n+1)}$ on $U$,
where $\omega''=\omega-\omega^\g$  is the conformal vector of $U$.
It follows that 
the form $(~|~)$ restricts to a symmetric invariant bilinear form of  $U$.

Now since $V$ is
non-negatively graded and each $V_d$ is finite-dimensional,
 $V$ belongs to $\on{KL}_k$
as a $\affg$-module.
Also,
by the assumption (2),
the $\affg$-module
$V^k(\g)$ is projective in $\on{KL}_k$
and we have
\begin{align*}
\on{Hom}_{\widehat{\g}}(V^k(\g),M)=M_{[0]}=M_{[0]}^{gen}
\end{align*}
for $M\in \on{KL}_k$,
where 
$M_{[\lam]}$  and
$M_{[\lam]}^{gen}$  are the $L_0^\g$-eigenspace 
and the 
$L_0^\g$-generalized eigenspace
of $M$ of eigenvalue $\lam$,
respectively,
see e.g.
\cite[Lemma 2.1]{ACL17}.
Therefore  
$U=V^\g\cong \on{Hom}_{\widehat{\g}}(V^k(\g), V)$
equals to $V_{[0]}\subset V$.
%
Since $(L_0^\g v,w)=(v, L_0^\g w)$,
the restriction of the form $(~|~)$ to
$U=V_{[0]}
$ is
 non-generate.
This implies that  $U$
 is simple as required.

\end{proof}

\begin{theorem}[\cite{ACL17}]
\label{theorem:ACL}
Let $\g$ be simply-laced.
Assume that $k+h^{\vee}-1\not\in \Q_{\leq 0}$,
and
define $\ell$
by the formula
\begin{equation*}
\frac{1}{k+h^\vee}+\frac{1}{\ell+h^\vee}=1.
\end{equation*}
 Then we have a vertex algebra isomorphism
\begin{align}
\W^{\ell}(\g)\cong \on{Com}(V^k(\g),V^{k-1}(\g)\* L_1(\g))=(V^{k-1}(\g)\* L_1(\g))^{\g[t]}.
\end{align}
Hence we have a 
 conformal embedding
\begin{align}
V^k(\g)\otimes \W^\ell(\g)\hookrightarrow 
V^{k-1}(\g)\otimes L_1(\g).
\label{eq:ACL}
\end{align}
\end{theorem}
 
The following assertion  gives a generalization of
\cite[Main Theorem 1]{ACL17}.
\begin{theorem}\label{theorem:coset-is-simple-1}
Let $\g$ be simply-laced and
suppose that $k+h^{\vee}-1\not\in \Q_{\leq 0}$.
Then
\begin{align*}
\W_{\ell}(\g)\cong (L_{k-1}(\g)\* L_1(\g))^{\g[t]}.
\end{align*}
In particular,
\eqref{eq:ACL} induces
  a vertex algebra homomorphism
\begin{align}
V^k(\g)\otimes \W_\ell(\g)\rightarrow 
L_{k-1}(\g)\otimes L_1(\g).
\label{eq:ACL2}
\end{align}
\end{theorem}
\begin{proof}
First,
 $L_{k-1}(\g)\* L_1(\g)$ is a simple conical conformal vertex algebra
 with conformal vector $\omega=\omega^\g\*1 +1\* \omega^\g$.
 Since  $\omega_{(2)}$ annihilates the weight one space
of $L_{k-1}(\g)\* L_1(\g)$,
$L_{k-1}(\g)\* L_1(\g)$ is equipped with a non-degenerate invariant symmetric bilinear form (\cite{Li94}).
Hence we immediately obtain from 
\Cref{thm:simplicity-of-cosets}
 that
the vertex algebra 
$(L_{k-1}(\g)\* L_1(\g))^{\g[t]}=\on{Com}(\varphi(V^k(\g)),L_{k-1}(\g)\* L_1(\g))$
is simple,
where $\varphi:V^k(\g)\ra L_{k-1}(\g)\* L_1(\g)$ is induced by the diagonal action of $\affg$ on $L_{k-1}(\g)\* L_1(\g)$.


On the other hand,
we have
by  \Cref{theorem:ACL}
that
$\W^{\ell}(\g)\cong 
(V^{k-1}(\g)\* L_1(\g))^{\g[t]}\cong \on{Hom}_{\widehat{\g}}(V^k(\g), V^{k-1}(\g)\* L_1(\g))
$.
The projectively of $V^k(\g)$ (see the proof of \Cref{thm:simplicity-of-cosets})
implies that 
the surjection 
$V^{k-1}(\g)\* L_1(\g)\ra L_{k-1}(\g)\* L_1(\g)$
 gives rise to the surjection
$\W^{\ell}(\g)\twoheadrightarrow\on{Hom}_{\widehat{\g}}(V^k(\g), L_{k-1}(\g)\* L_1(\g))
\cong  (L_{k-1}(\g)\* L_1(\g))^{\g[t]}
$.
It follows that
$ (L_{k-1}(\g)\* L_1(\g))^{\g[t]}$ is  isomorphic to the simple quotient
$\W_{\ell}(\g)$ of $\W^{\ell}(\g)$.

%
\end{proof}

Let
$T_{\lam,0}^{\ell}=H_{DS,f_{prin}}^0(V^k(\lam))$,
where $V^k(\lam)$
is the Weyl module of $\affg$ with highest weight $\lam+k\Lam_0$
and $H_{DS,f}^\bullet(?)$ is the quantized Drinfeld-Sokolov reduction associated with $f$.
We have 
$T^{\ell}_{\lam,0}=\bigoplus_{d\in h_{\lam}+\Z_{\geq 0}}(T^{\ell}_{\lam,0})_{[d]_W}$,
and 
  $[T^{\ell}_{\lam,0}: \mathbf{L}(\chi_{\mu})]\ne 0$
only if $[\mathbf{M}(\chi_{\lam}): \mathbf{L}(\chi_{\mu})]\ne 0$,
or equivalently,
 $\chi_{\mu}\preceq_{\ell} \chi_\lam$,
 see  \cite{AraFre19}.

By \cite[Main Theorem 3]{ACL17} and its proof,
we have the following assertion.
\begin{proposition}\label{prop:comp-factor1}
Suppose that $k+h^{\vee}-1\not\in \Q_{\leq 0}$.
We have
\begin{align*}
[V^{k-1}(\g)\otimes L_1(\g)]=\sum_{\lam \in P_+\cap Q}[V^k(\lam)\* T_{\lam,0}^{\ell}]
\end{align*}
in the Grothendieck group $K_0(V^k(\g)\otimes \W^\ell(\g)\on{-mod})$
of the category $V^k(\g)\otimes \W^\ell(\g)\on{-mod}$ of $V^k(\g)\otimes \W^\ell(\g)$-modules.
\end{proposition}

\begin{theorem}\label{th:simplicity-of-L}
Let $\g$ be simply-laced,
$k+h^{\vee}-1\not\in \Q_{\leq 0}$.
Suppose that
if $\chi_{\lam}\sim_{\ell} \chi_{0}$ 
 for  $\lam\in P_+\cap Q$
 then
 $h_{\lam}\geq h_0=0$
 and the equality holds if and only if $\lam=0.$
  Then the vertex algebra homomorphism
  \eqref{eq:ACL2}
factors through the embedding
\begin{align}
L_k(\g)\* \W_{\ell}(\g)\hookrightarrow L_{k-1}(\g)\* L_1(\g).
\end{align}
Moreover,
$L_k(\g)$ and $ \W_{\ell}(\g)$ form a dual pair in $L_{k-1}(\g)\* L_1(\g)$.
\end{theorem}
\begin{proof}
Set  $V=V^{k-1}(\g)\* L_1(\g)$,
$L=L_{k-1}(\g)\* L_1(\g)$.
The action of the Sugawara conformal vector of $V^k(\g)$
gives the decomposition
$V=\bigoplus_{d\in \C}V_{[d]}$,
$L=\bigoplus_{d\in \C}L_{[d]}$,
where $M_{[d]}$ is the generalized $L_0^{\g}$ -eigenspace
with eigenvalue $d$. 
Each $V_{[d]}$ and $L_{[d]}$ are $\W^{\ell}(\g)$-submodules,
which belongs to $\mc{O}(\W^{\ell}(\g))$ by \Cref{prop:comp-factor1}.
Hence the multiplicities
$[V_{[d]}: \mathbf{L}_{\ell}(\chi_{\mu})]$
and $[L_{[d]}: \mathbf{L}_{\ell}(\chi_{\mu})]$ 
are well-defined.

Consider the multiplicity
of $\W_{\ell}(\g)=\mathbf{L}_{\ell}(\chi_0)$ in $V_d$.
By the assumption,
for  $\lam\in P_+\cap Q$,
 $[T^{\ell}_{\lam,0}:\mathbf{L}_{\ell}(\chi_0)]\ne 0$  if only if 
$\lam=0$.
Since
$[T^{\ell}_{0,0}:\mathbf{L}_{\ell}(\chi_0)]=1$,
it follows from \Cref{prop:comp-factor1} that 
$[V_{[d]}:\mathbf{L}_{\ell}(\chi_0)]\leq 
\dim_{\C}  V^k(\g)_{[d]}$.
On the other hand,
the injectivity of the map \eqref{eq:ACL}
implies that $[V_{[d]}:\mathbf{L}_{\ell}(\chi_0)]\geq 
\dim_{\C}  V^k(\g)_{[d]}$.
We conclude that  
\begin{align}
[V_{[d]}:\mathbf{L}_{\ell}(\chi_0)]=\dim_{\C}  V^k(\g)_{[d]}
\label{eq:mutilplicity-of-vaccume}
\end{align}
for all $d$.

Consider the image
$V_{[d]}^{[\chi_0]}$ of the projection of $V_{[d]}$
to the block $\mc{O}(\W^{\ell}(\g))^{[\chi_0]}$,
$V^{[\chi_0]}=\bigoplus_d V_{[d]}^{[\chi_0]}$.
By the assumption,
we have
$\dim (V_{[d]}^{[\chi_0]})_{[0]_W}^{gen}=[V_{[d]}:\mathbf{L}_{\ell}(\chi_0)]$.
Since $L_0^{W}$ acts semisimply 
on the image of \eqref{eq:ACL},
we obtain from \eqref{eq:mutilplicity-of-vaccume} that
\begin{align*}
 (V^{[\chi_0]})_{[0]_W}^{gen}= (V^{[\chi_0]})_{[0]_W}\cong  V^k(\g).
\end{align*}
It follows that
\begin{align*}
(L^{[\chi_0]})_{[0]_W}^{gen}= (L^{[\chi_0]})_{[0]_W},
\end{align*}
where
$L^{[\chi_0]}=\bigoplus_d L^{[\chi_0]}_{[d]}$.

Clearly,
the restriction of the non-degenerate invariant form
of $L$
to  the subspace $(L^{[\chi_0]})_{[0]_W}^{gen}$
is non-degenerate.
Since it is a quotient of $V^k(\g_0)$,
we obtain that
\begin{align*}
(L^{[\chi_0]})_{[0]_W}^{gen}\cong L_k(\g).
\end{align*}
In other words,
the image of
the homomorphism
$V^k(\g)\ra L$
induced by the diagonal action of $\affg$ is simple.
We have shown that the vertex algebra homomorphism
  \eqref{eq:ACL2}
factors through the embedding
$L_k(\g)\* \W_{\ell}(\g)\hookrightarrow L
$.

 Finally, 
 we have
 \begin{align*}
 \on{Com}(\W_{\ell}(\g), L)
 \cong \Hom_{\W^{\ell}(\g)}(\W^{\ell}(\g), L)
=\Hom_{\W^{\ell}(\g)}(\W^{\ell}(\g), L^{[\chi_0]})
 \cong (L^{[\chi_0]})_{[0]_W},
 \end{align*}
 where the last isomorphism follows from the assumption.
 Therefore we obtain that $ \on{Com}(\W_{\ell}(\g), L)
\cong L_k(\g)$, and this completes the proof.
\end{proof}

\begin{theorem}[\cite{AraCreFei22}]\label{thm:ACF}
Let $f$ be a nilpotent element of $\g$ and $n$ a non-negative integer.
For any quotient vertex algebra $V$ of $V^k(\g)$,
we have a vertex algebra isomorphism
\[
H_{DS,f}^0(V\otimes L_n(\g))\cong H_{DS,f}^0(V)\otimes L_n(\g).
\]
Moreover, for any $V$-module $M$ and $L_n(\g)$-module $N$, there is
an isomorphism
$$
H_{DS,f}^i(M\otimes N)\cong H_{DS,f}^i(M)\otimes N
$$
of $H_{DS,f}^0(V)\otimes L_n(\g)$-modules for all $i\in\Z$.
\end{theorem}

%
%


\begin{theorem}\label{th:simplicity-of-W}
Let $\g$ be simply-laced.
Suppose 
\begin{enumerate}
\item $k\not\in \Z_{\geq 0}$,
$k+h^{\vee}-1\not\in \Q_{\leq 0}$;
\item if $\chi_{\lam}\sim_{\ell} \chi_{0}$ 
 for  $\lam\in P_+\cap Q$
 then
 $h_{\lam}\geq h_0=0$
 and the equality holds if and only if $\lam=0$;
 \item $2\rho-(k+h^{\vee})\theta\in \sum_{\alpha\in \Delta_+}\mathbb{R}_{\geq 0}\alpha$.
\end{enumerate} 
  Then the map
  \eqref{eq:ACL2}
induces  an embedding
\begin{align}
\W_k(\g,f_{\theta})\* \W_{\ell}(\g)\hookrightarrow \W_{k-1}(\g,f_{\theta})\* L_1(\g)
\end{align}
of vertex algebras.
Moreover,
$\W_k(\g,f_{\theta})$ and $ \W_{\ell}(\g)$ form a dual pair in $\W_{k-1}(\g,f_{\theta})\* L_1(\g)$.
\end{theorem}
\begin{proof}
By \cite{Ara05},
the functor
$H_{DS,f_{\theta}}^0(?)$ is exact and 
$H_{DS,f_{\theta}}^0(L_k(\g))\cong \W_k(\g,f_{\theta})$ 
provided that  $k\not \in \Z_{\geq 0}$.
Therefore by  applying $H_{DS,f_{\theta}}^0(?)$ 
to  \eqref{eq:ACL2} with respect to the level $k$ action of $\affg$,
we obtain an embedding
\begin{align}
\W_k(\g,f_{\theta})\* \W_{\ell}(\g)\hookrightarrow H_{DS,f_{\theta}}^0(L_{k-1}(\g)\* L_1(\g))\cong
\W_{k-1}(\g,f_{\theta})\* L_1(\g),
\label{eq:minimal-simple-embeddding}
\end{align}
where the last isomorphism follows from
 \Cref{thm:ACF}.
Also,
the proof of  \Cref{th:simplicity-of-L}
shows that
\begin{align*}
\on{Com}(\W_k(\g),\W_{k-1}(\g,f_{\theta})\* L_1(\g) )
\cong \Hom (\W^\ell(\g),\W_{k-1}(\g,f_{\theta})\* L_1(\g))\\
\cong (\W_{k-1}(\g,f_{\theta})\* L_1(\g))^{[\chi_0]}_{[0]_W}
 \cong H_{DS,f_{\theta}}^0((V_{k-1}(\g,f_{\theta})\* L_1(\g))^{[\chi_0]}_{[0]_W})\\
 \cong H_{DS,f_{\theta}}^0(L_{k}(\g))\cong \W_k(\g,f_{\theta}).
\end{align*}
It remains to show that
\begin{align}
\on{Com}(\W_{\ell}(\g,f_{\theta}),\W_{k-1}(\g,f_{\theta})\* L_1(\g))\cong 
 \W_{\ell}(\g).
\end{align}
Note that
$\on{Com}(\W_{\ell}(\g,f_{\theta}),\W_{k-1}(\g,f_{\theta})\* L_1(\g)) \cong 
\Hom_{\W^k(\g,f_{\theta})} (\W^k(\g,f_{\theta}),\W_{k-1}(\g,f_{\theta})\* L_1(\g))$.

Similarly as in \eqref{eq:minimal-simple-embeddding},
the map
\eqref{eq:ACL}
induces an embedding
\begin{align}
\W^k(\g,f_{\theta})\* \W^{\ell}(\g)\hookrightarrow \W^{k-1}(\g,f_{\theta})\* L_1(\g)
\label{eq:minimal-universal-embeddding}
\end{align}
and hence 
we can regard 
$\W^{k-1}(\g,f_{\theta})\* L_1(\g)$  as a 
$\W^k(\g,f_{\theta})$-module.
Clearly,
each eigenspace
$(\W^{k-1}(\g,f_{\theta})\* L_1(\g))_{[d]_{min}}$ is a $\W^{\ell}(\g)$-submodule.

Let $\omega_{min}$ be the conformal vector of the 
minimal W-algebra $\W^k(\g,f_{\theta})$,
$L^{min}(z)=\sum_{n\in \Z}L^{min}_nz^{-n-2}=Y(\omega_{min},z)$.
Let $M_{[d]_{min}}^{gen}$ and $M_{[d]_{min}}$ be
the generalized $L_0^{min}$-eigenspace 
and the  $L_0^{min}$-eigenspace of eigenvalue $d$
of a $\W^{k}(\g,f_{\theta})$-module $M$.

We claim that
\begin{align}
&\W^{k-1}(\g,f_{\theta})\* L_1(\g)=\bigoplus_{d\geq 0}
(\W^{k-1}(\g,f_{\theta})\* L_1(\g))^{gen}_{[d]_{min}},
\label{claim-for-minimal1}\\
&(\W^{k-1}(\g,f_{\theta})\* L_1(\g))^{gen}_{[0]_{min}}=
(\W^{k-1}(\g,f_{\theta})\* L_1(\g))_{[0]_{min}}\cong \W^{\ell}(\g).
\label{claim-for-minimal2}
\end{align}
To see this,
note that 
by  \Cref{prop:comp-factor1}
\begin{align*}
[\W^{k-1}(\g,f_{\theta})\otimes L_1(\g)]=\sum_{\lam \in P_+\cap Q}[H_{DS,f_{\theta}}^0(V^k(\lam))\* T_{\lam,0}^{\ell}]
\end{align*}
in the Grothendieck group $K_0(\W^k(\g,f_{\theta})\otimes \W^\ell(\g)\on{-mod})$
of the category $\W^k(\g,f_{\theta})\otimes \W^\ell(\g)\on{-mod}$ of $\W^k(\g,f_{\theta})\otimes \W^\ell(\g)$-modules.
We have
\begin{align*}
H_{DS,f_{\theta}}^0(V^k(\lam))=\bigoplus_{d\in h_{\lam}^{min}+1/2\Z_{\geq 0}}H_{DS,f_{\theta}}^0(V^k(\lam))_{[d]_{min}},
\end{align*}
where
\begin{align*}
h_{\lam}^{min}=\frac{(\lam|\lam+2\rho)}{2(k+h^{\vee})}-\frac{1}{2}(\lam|\theta)
=\frac{1}{2(k+h^{\vee})}(|\lam|^2+(\lam|2\rho-(k+h^{\vee})\theta)).
\end{align*}
Thus
the assumption (3)
implies that,
 for $\lam\in P_+$, 
 we have
$h_{\lam}\geq 0$  and the quality holds 
if and only if $\lam=0$.
This shows 
\eqref{claim-for-minimal1}.
Moreover, since $H_{DS,f_{\theta}}^0(V^k(0))_{[0]_{min}}=\W^k(\g,f_{\theta})_{[0]_{min}}=\C$,
it follows that
\eqref{eq:minimal-universal-embeddding}
restricts to an isomorphism
$\W^{\ell}(\g)\cong  (\W^{k-1}(\g,f_{\theta})\otimes L_1(\g))_{[0]_{min}}^{gen}$
and that 
$ (\W^{k-1}(\g,f_{\theta})\otimes L_1(\g))_{[0]_{min}}^{gen}
= (\W^{k-1}(\g,f_{\theta})\otimes L_1(\g))_{[0]_{min}}
$.

Since $\W_{k-1}(\g,f_{\theta})\* L_1(\g)$ is a quotient of 
$\W^{k-1}(\g,f_{\theta})\* L_1(\g)$,
we have by \eqref{claim-for-minimal1} and \eqref{claim-for-minimal2}
that
\begin{align*}
&\W_{k-1}(\g,f_{\theta})\* L_1(\g)=\bigoplus_{d\geq 0}
(\W_{k-1}(\g,f_{\theta})\* L_1(\g))^{gen}_{[d]_{min}},\\
&(\W_{k-1}(\g,f_{\theta})\* L_1(\g))^{gen}_{[0]_{min}}=
(\W_{k-1}(\g,f_{\theta})\* L_1(\g))_{[0]_{min}},
\end{align*}
and that
$(\W_{k-1}(\g,f_{\theta})\* L_1(\g))_{[0]_{min}}$ is a quotient of $\W^{\ell}(\g)$.
This gives that
\begin{align*}
\Hom (\W^k(\g,f_{\theta}),\W_{k-1}(\g,f_{\theta})\* L_1(\g))
\cong (\W_{k-1}(\g,f_{\theta})\* L_1(\g))_{[0]_{min}}^{gen}.
\end{align*}
On the other hand,
the vertex algebra
$\W_{k-1}(\g,f_{\theta})\* L_1(\g)$ is equipped with a
non-degenerate invariant form (\cite{AEkeren19}).
Since 
the restriction of the non-degenerate invariant form on 
$\W_{k-1}(\g,f_{\theta})\* L_1(\g)$ 
to the subspace 
$ (\W_{k-1}(\g,f_{\theta})\* L_1(\g))_{[0]_{min}}^{gen}$ is  non-degenerate,
we obtain that
$(\W_{k-1}(\g,f_{\theta})\* L_1(\g))_{[0]_{min}}^{gen}\cong \W_{\ell}(\g)$.
This completes the proof.
 
\end{proof}

It is clear that  \Cref{sec:dualsimple}
follows from \Cref{th:simplicity-of-W}
and the following assertion.
\begin{lemma}
The three conditions in \Cref{th:simplicity-of-W}
are satisfied for $k$ as in \Cref{sec:dualsimple}.
\end{lemma}
\begin{proof}
The condition (1) is clearly satisfied
and it is straight forward to check that the condition (3) is satisfied
(see Tables \ref{table:D4}-\ref{table:E8}).
Let us show that
the condition (2) is satisfied.
Set $\Lam=\lam-(\ell+h^{\vee})\rho^{\vee}+\ell\Lam_0$.
Then
$$\Lam+\hat\rho=-\frac{1}{p-1}\rho+\frac{p}{p-1} \Lam_0,$$
where $\hat\rho=\rho+h^{\vee}\Lam_0$.
Let
$\hat{\Delta}(\Lam)=\{\alpha\in \widehat{\Delta}^{re}\mid \langle \Lam+\hat\rho, \alpha^{\vee}\rangle
\in \Z\}$,
the set of integral roots of $\Lam$,
where $ \widehat{\Delta}^{re}$ is the set of real roots of $\affg$.
One finds that a base $\hat{\Pi}(\Lam)$
of $\hat{\Delta}(\Lam)$ is given by
\begin{align*}
\hat{\Pi}(\Lam)=y(S_{(p-1)}),
\end{align*}
where $y$ is an element of the extended Weyl group of $\widehat{\g}$ given in Tables \ref{table:D4}-\ref{table:E8}
and
\begin{align*}
S_{(q)}=\{\alpha_1,\dots,\alpha_{\on{rk}\g}, -\theta+q\delta\}.
\end{align*}
Here we have adopted the standard Bourbaki numbering for the simple roots,
$\theta$ is the highest root of $\g$,
$s_\alpha$ is the reflection corresponding to $\alpha$, and $s_i=s_{\alpha_i}$.
By definition
$\chi_{\mu}\sim_{\ell}\chi_0$ if and only if
$\mu-(\ell+h^{\vee})\rho^{\vee}+\ell\Lam_0\in \widehat{W}(\Lam)\circ \Lam$,
where $\widehat{W}(\Lam)=\bra s_{\alpha}\mid \alpha \in \widehat{\Delta}(\Lam)\ket$ is integral Weyl group of $\Lam$.

Let $\beta$ be the element of $\hat{\Pi}(\Lam)$ given in Tables \ref{table:D4}-\ref{table:E8}.
Then
\begin{align}
\langle \Lam+\hat{\rho},\beta^{\vee}\rangle=-1\text{ and }
\langle \Lam+\hat{\rho},\gamma^{\vee}\rangle \geq 1\text{ for }\gamma\in \hat{\Pi}(\Lam)\backslash{\beta},
\label{eq:dominant-principal-finding1}\\
\langle \beta, D\rangle=0\text{ and }\langle \gamma, D\rangle >0\text{ for }\gamma\in \hat{\Pi}(\Lam)\backslash{\beta},
\label{eq:dominant-principal-finding2}
\end{align}
It follows \eqref{eq:dominant-principal-finding1} that
$$\Lam_+:=s_{\beta}\circ \Lam\equiv \beta+\Lam\pmod{\delta}$$ 
is the unique dominant weight in the orbit $\widehat{W}(\Lam)\circ \Lam$,
and that
$\chi_{\beta+(\ell+h^{\vee})\rho^{\vee}}$
is the only central character $\chi_{\lam}$ such that $\chi_{\lam}\succeq_{\ell} \chi_0$,
 $\chi_{\lam}\ne  \chi_0$.
It follows from  \eqref{eq:conformal-dim-comparison} and \eqref{eq:dominant-principal-finding2}
that
 $h_{\lam}\geq 0$ for all $\chi_{\lam}$ such that $\chi_{\lam}\sim_{\ell}\chi_0$,
and the equality holds if and only if $\chi_{\lam}=\chi_0, \chi_{\beta+(\ell+h^{\vee})\rho^{\vee}}$.
 Since  $\beta+(\ell+h^{\vee})\rho^{\vee}\not \in W\circ P_+$,
 we conclude that the condition (2) is satisfied as required.
  \end{proof}

\begin{table}[hbtp]
  \caption{$d_4$}
  \label{table:D4}
  \begin{tabular}{c|c|c|c}
    $m$  & $p$ &  $y$ &$\beta$ \\
    \hline 
    0 & $5$  & $t_{-\rho}s_{-\theta+4\delta}s_2$& $y(\alpha_2)=\theta-\alpha_2$     
  \end{tabular}
   \begin{align*}
2\rho=6\alpha_1+10\alpha_2+6\alpha_3+6\alpha_4,\quad
\theta=\alpha_1+2\alpha_2+\alpha_3+\alpha_4.
  \end{align*}
\end{table}

\begin{table}[hbtp]
  \caption{$e_6$}
  \label{table:E6}
  \begin{tabular}{c|c|c|c}
    $m$  & $p$ &  $y$ &$\beta$ \\
    \hline 
    0 & $10$  & $t_{-\rho}s_{-\theta+9\delta}s_2s_4$& $y(\alpha_4)=\theta-\alpha_2-\alpha_4$    \\
        1 & $11$  & $t_{-\rho}s_{-\theta+10\delta}s_2$& $y(\alpha_2)=\theta-\alpha_2$   
  \end{tabular}
  \begin{align*}
&2\rho=16\alpha_1+22\alpha_2+30\alpha_3+42\alpha_4+30\alpha_5+16\alpha_6,\\
&\theta=\alpha_1+2\alpha_2+2\alpha_3+3\alpha_4+2\alpha_5+\alpha_6.
  \end{align*}
\end{table}

\begin{table}[hbtp]
  \caption{$e_7$}
  \label{table:E7}
  \begin{tabular}{c|c|c|c}
    $m$  & $p$ &  $y$ &$\beta$ \\
    \hline 
    0 & $15$  & $t_{-\rho}s_{-\theta+14\delta}s_1s_3s_4$& $y(\alpha_4)=\theta-\alpha_1-\alpha_2-\alpha_3$    \\
        1 & $16$  & $t_{-\rho}s_{-\theta+15\delta}s_1s_3$& $y(\alpha_3)=\theta-\alpha_1-\alpha_2$   \\
          2 & $17$  & $t_{-\rho}s_{-\theta+16\delta}s_1$& $y(\alpha_1)=\theta-\alpha_1$   
  \end{tabular}
  \begin{align*}
&2\rho=34\alpha_1+49\alpha_2+66\alpha_3+ 96\alpha_4+ 75\alpha_5+ 52\alpha_6+ 27\alpha_7,\\
&\theta=2\alpha_1+ 2\alpha_2+ 3\alpha_3+ 4\alpha_4+ 3\alpha_5+ 2\alpha_6 +\alpha_7.
  \end{align*}
\end{table}

\begin{table}[hbtp]
  \caption{$e_8$}
  \label{table:E8}
  \begin{tabular}{c|c|c|c}
    $m$  & $p$ &  $y$ &$\beta$ \\
    \hline 
    0 & $25$  & $t_{-\rho}s_{-\theta+24\delta}s_8s_7s_6s_5s_4$& $y(\alpha_4)=\theta-\alpha_4-\alpha_5-\alpha_6-\alpha_7-\alpha_8$    \\
        1 & $26$  & $t_{-\rho}s_{-\theta+25\delta}s_8s_7s_6s_5$& $y(\alpha_5)=\theta-\alpha_5-\alpha_6-\alpha_7-\alpha_8$   \\
          2 & $27$  & $t_{-\rho}s_{-\theta+26\delta}s_8s_7s_6$& $y(\alpha_6)=\theta-\alpha_6-\alpha_7-\alpha_8$   \\
             3 & $28$  & $t_{-\rho}s_{-\theta+27\delta}s_8s_7$& $y(\alpha_7)=\theta-\alpha_7-\alpha_8$   \\
               4 & $29$  & $t_{-\rho}s_{-\theta+28\delta}s_8$& $y(\alpha_8)=\theta-\alpha_8$
  \end{tabular}
    \begin{align*}
&2\rho=92\alpha_1+ 136\alpha_2+ 182\alpha_3+ 270\alpha_4+ 220\alpha_5+ 168\alpha_6+ 114\alpha_7+ 58\alpha_8,\\
&\theta=2\alpha_1+ 2\alpha_2+ 4\alpha_3+ 6\alpha_4+ 5\alpha_5+ 4\alpha_6 +3\alpha_7+2\alpha_8.
  \end{align*}
\end{table}

\bibliographystyle{alpha}
\bibliography{/Users/tomoyuki/Documents/Dropbox/bib/math}

\end{document}